\documentclass[11pt]{article}
\usepackage[pagebackref=true,colorlinks,linkcolor=blue,citecolor=magenta]{hyperref}
\usepackage{amssymb,amsmath}
\usepackage{amsmath}
\usepackage{dsfont}
\usepackage{amsfonts}
\newtheorem{precor}{{\bf Corollary}}

\newenvironment{cor}{\begin{precor}{\hspace{-0.5
               em}{\bf.\ }}}{\end{precor}}
\newtheorem{precon}{{\bf Conjecture}}

\newtheorem{predefin}{{\bf Definition}}

\newenvironment{defin}[1]{\begin{predefin}{\hspace{-0.5
                   em}{\bf.\ }}{\rm
#1}\hfill{$\spadesuit$}}{\end{predefin}}
\newtheorem{preexm}{{\bf Example}}

\newtheorem{prerem}{{\bf Remark}}

\newtheorem{preappl}{{\bf Application}}

\newtheorem{prelem}{{\bf Lemma}}

\newenvironment{lem}{\begin{prelem}{\hspace{-0.5
               em}{\bf.\ }}}{\end{prelem}}
\newtheorem{preproof}{{\bf Proof.\ }}

\newenvironment{proof}[1]{\begin{preproof}{\rm
               #1}\hfill{$\blacksquare$}}{\end{preproof}}
\newtheorem{presproof}{{\bf Sketch of Proof.\ }}

\newtheorem{prethm}{{\bf Theorem}}

\newenvironment{thm}{\begin{prethm}{\hspace{-0.5
               em}{\bf.\ }}}{\end{prethm}}
\newtheorem{prealphthm}{{\bf Theorem}}

\newenvironment{alphthm}{\begin{prealphthm}{\hspace{-0.5
               em}{\bf.\ }}}{\end{prealphthm}}
\newtheorem{prepro}{{\bf Proposition}}

\newenvironment{pro}{\begin{prepro}{\hspace{-0.5
               em}{\bf.\ }}}{\end{prepro}}
\newtheorem{preprb}{{\bf Problem}}

\def\conct[#1,#2]{\mbox {${#1} \leftrightarrow {#2}$}}
\def\dconct[#1,#2]{\mbox {${#1} \rightarrow {#2}$}}
\def\deg[#1,#2]{\mbox {$d_{_{#1}}(#2)$}}
\def\mindeg[#1]{\mbox {$\delta_{_{#1}}$}}
\def\maxdeg[#1]{\mbox {$\Delta_{_{#1}}$}}
\def\outdeg[#1,#2]{\mbox {$d_{_{#1}}^{^+}(#2)$}}
\def\minoutdeg[#1]{\mbox {$\delta_{_{#1}}^{^+}$}}
\def\maxoutdeg[#1]{\mbox {$\Delta_{_{#1}}^{^+}$}}
\def\indeg[#1,#2]{\mbox {$d_{_{#1}}^{^-}(#2)$}}
\def\minindeg[#1]{\mbox {$\delta_{_{#1}}^{^-}$}}
\def\maxindeg[#1]{\mbox {$\Delta_{_{#1}}^{^-}$}}

\def\dre[#1,#2,#3]{\mbox {${\cal E}_{_{#3}}(#1,#2)$}}
\def\pdre[#1,#2,#3]{\mbox {${\cal P}_{_{#3}}(#1,#2)$}}
\def\var[#1,#2]{\mbox {${\rm Var}_{_{#1}}(#2)$}}
\def\ls[#1]{\mbox {$\xi^{^{#1}}$}}
\def\hom[#1,#2]{\mbox {${\rm Hom}({#1},{#2})$}}
\def\onvhom[#1,#2]{\mbox {${\rm Hom^{v}}(#1,#2)$}}
\def\onehom[#1,#2]{\mbox {${\rm Hom^{e}}(#1,#2)$}}
\def\core[#1]{\mbox {$#1^{^{\bullet}}$}}
\def\cay[#1,#2]{\mbox {${\rm Cay}({#1},{#2})$}}
\def\cays[#1,#2]{\mbox {${\rm Cay_{s}}({#1},{#2})$}}
\def\dirc[#1]{\mbox {$\stackrel{\rightarrow}{C}_{_{#1}}$}}
\def\cycl[#1]{\mbox {${\bf Z}_{_{#1}}$}}

\begin{document}
\begin{center}
{\Large \bf A Generalization of Cover Free Families}\\
\vspace*{0.5cm}
{\bf Mehdi Azadi motlagh $^\ast$ and Farokhlagha Moazami$^\dag$}\\
{\it $^\ast$Department of Mathematics}\\
{\it Kharazmi University, 50 Taleghani Avenue, 15618, Tehran, Iran}\\
{\tt std$\_$m.azadim@khu.ac.ir}\\
{\it $^\dag$ Cyberspace Research Center} \\
{\it Shahid Beheshti University, G.C.}\\
{\it P.O. Box {\rm 1983963113}, Tehran, Iran}\\
{\tt f$\_$moazemi@sbu.ac.ir}\\ \ \\
\end{center}
\begin{abstract}
An $(r,w; d)$ -cover-free family (CFF) is a family of subsets of a finite set such that the intersection of any $r$ members of the family contains at least $d$ elements that are not in the union of any other $w$ members. The minimum number of elements for which there exists an $(r,w; d)$-CFF with $t$ blocks is denoted by 
$N((r,w;d), t)$. In this paper, we determine the exact value of $N((r,w;d), t)$  for some special parameters. Also, we present two constructions for $(2,1;d) $-CFF and  $(2,2;d) $-CFF which improve the existing constructions. Moreover, we introduce a generalization of cover-free families which is motivated by an application of CFF in the key pre-distribution schemes. Also, we investigate some properties and bounds on the parameters of this generalization.

\begin{itemize}
\item[]{{\footnotesize {\bf Key words:}\ Biclique covering number, Cover-free families, Key pre-distribution.}}
\item[]{ {\footnotesize {\bf Subject classification:} 05B40.}}
\end{itemize}
\end{abstract}
\section{Introduction}
A {\it key pre-distribution scheme} (KPS) is a method by which a trusted authority (TA) distributes secret information among a set of users in such a way that every user in a group in some specified family of privileged subsets is able to compute a common key associated with that group. This common key must remain unknown to some specified coalitions of users (forbidden subsets) outside the privileged group. To construct a key pre-distribution scheme, TA can  
 use a $\{0,1\}$-matrix $M$ that is called {\it key distribution pattern}. A key distribution pattern specifies which users are to receive which keys. Namely, user $u_i$ is given the key $k_j$ if and only if $M[i,j]=1$. Mitchell and Piper considered a key distribution pattern in which there is a key for every group of $r$ users, such that this key is secure against any disjoint coalition of at most $w$ users. A family of sets is called an $(r,w)$-{\it cover-free family} if no intersection of $r$ sets of the family are covered by a union of
any other $w$ sets of the family. Easily one can see that Mitchell-Piper key distribution patterns are equivalent to $(r,w)$-cover-free families.
Cover-free families were first introduced in $1964$ by Kautz and Singleton in the context of superimposed binary codes ~\cite{1053689}. Cover-free families (CFFs) were considered from different subjects such as combinatorics, information theory and group testing by many researchers (see, for example, ~\cite{erdos1,rcff2,haji2, Hajiabolhassan20123626,JCD:JCD10056,rcff3,Stinson2004463,wei}). Stinson  and Wei~\cite{Stinson2004463} have introduced a generalization of cover-free families as follows. 
\begin{defin}{ Let $d, n, t, r,$ and $ w $ be positive integers  and $ B = \{ B_1, \ldots, B_t \}$ be a collection of subsets of a set $X$, where
$|X| = n$. Each element of the collection $B$ is called a block and the elements of $X$ are called points. The pair $(X,B)$ is called an $ (r,w; d)-CFF(n, t) $ if for any two sets
of indices $ L,M \subseteq [t] $ such that $ L \cap M = \emptyset , |L| = r, $ and $ |M| = w, $ we have
$$ \vert (\bigcap_{l \in L}^{} B_{l}) \setminus (\bigcup_{m \in M}^{} B_{m}) \vert \geq d .$$
Let $N((r,w;d),t)$ denote the minimum number of points of $X$ in an  $(r,w;d)-CFF$ having $t$ blocks and  $T((r,w;d),n)$ denote the maximum number of blocks in an $(r,w;d)-CFF$ with $n$ points.
} 
\end{defin}
An $(r,w;d)$-cover-free family yields  a key pre-distribution scheme  in which  every group of $r$ users have at least $d$ common keys that these keys are secure against any disjoint coalition of at most $w$ users. The case $r= 2$ is of particular interest, because this is the case where keys are associated with pairs of users. Also, in the most applications, we do~not need every $2$-subset of users have a common key. In other word, for every key pre-distribution scheme that privileged subsets are $2$-subsets, we can assign a graph as follows. Let the vertices of this graph be the users or nodes of the network and two vertices are adjacent if and only if these users or nodes can establish a common key. We name this graph the scheme graph. For instance, the scheme graph of key pre-distribution scheme constructed from a $(2,w;d)$-cover free family is a complete graph.
The definition of scheme graph leads us to the following  generalization of cover-free families.
\begin{defin}{ Let $d, n, t,$ and $w$ be positive integers and $ B = \{ B_1, \ldots, B_t \}$ be a collection of subsets of a set $X$, where
$|X| = n$. Assume that  $G$ is a graph with $V(G)=[t]$. 
 The pair $(X,B)$, is called a $ (G,w; d)-CFF(n, t) $ if for any two sets
of indices $ L,M \subseteq [t] $ such that $ L \cap M = \emptyset , |L| = 2, $ and $L$ is an edge of $G$ and $ |M| = w, $ we have
$$ \vert (\bigcap_{l \in L}^{} B_{l}) \setminus (\bigcup_{m \in M}^{} B_{m}) \vert \geq d .$$
}
\end{defin}
This definition has a dual form as follows.
\begin{defin} { Let $G$ be a graph and ${\cal A}= \{A_1,A_2,\ldots, A_n\}$ be a collection of subsets of $V(G)$. The  collection ${\cal A}$ is called a $(w,d)$-{\it covering} of $G$ if for every edge $\{u,v \}$ of the graph $G$ and every $w$-subset, $W \subseteq V(G)$, disjoint from $\{u,v\}$, there exist at least  $d$ sets $A_{j_1},\ldots, A_{j_d}\in {\cal A}$ such that  $\{u,v\} \subseteq A_{j_i}$ and $ W \cap A_{j_i} = \varnothing$, for any $1\leq i\leq d$. The number of sets in the minimum covering of $G$ is called the {\it key  pool number} of $G$ and is denoted by $N(G,w;d)$. This parameter is denoted by $N(G)$  whenever  $w=d=1$.}
\end{defin}
This generalization is a natural generalization for the set systems. For instance, Bollob{\'a}s and Scott~\cite{bllobas1} consider a generalization like this for separating systems.
In the following, we give a brief outline of graph theory which we need it. Throughout this paper, we only consider finite simple graphs. For a graph $ G $, let $V(G)$ and $E(G)$ denote its vertex and edge sets, respectively.  In this paper, by $[n]$, we shall mean the set $\{1,2, \ldots, n\}$. The {\it biclique covering number} $bc(G)$  of a graph $G$ is the smallest number of bicliques (complete bipartite subgraphs) of $G$ such that every edge of $G$ belongs to at least one of these bicliques. In the same manner, we can define $d$-{\it biclique covering number}  $bc_{d}(G)$ of a graph  $G$ which is the smallest number of bicliques of $G$ such that every edge of $G$ belongs to at least $d$ of these bicliques. In these  cases, when the bicliques are required to be edge-disjoint, the corresponding measures are known as the {\it biclique partition number} and $d$-{\it biclique partition number}   and are denoted by $bp(G)$ and $bp_{d}(G)$, respectively. Hajiabolhassan and Moazami ~\cite{Hajiabolhassan20123626} showed that the existence of an $(r,w;d)$-cover-free family results from the existence of $d$-biclique cover of bi-intersection graph  and vice versa. The {\it bi-intersection graph} $ I_{t}(r, w) $ is  a bipartite graph whose vertices are all $w$- and $r$-subsets of a  $t$-element set, where a $w$-subset is adjacent to an $r$-subset if and only if their intersection is empty. They \cite{haji2} also showed  the existence of a secure frame proof code results from the existence of biclique cover of Kneser graph  and vice versa. The {\it Kneser graph} $KG(t, r)$ is a graph whose vertices are all $r$-subsets of a $t$-element set, where two $r$-subsets are adjacent if and only if their intersection is empty. Motivated by these observations, we determine the exact value of $d$-biclique covering number of bi-intersection graph and Kneser graph for some special parameters in Section~\ref{bc and bp section}.  For many  applications of cover-free families, construction of cover-free families has been studied in the literature~\cite{li,frame,stinson1998}. In Section~\ref{bc and bp section}, we  also give a construction for $(2,1;d)$-cover-free family and a construction for $(2,2;d)$-cover-free family which improve the existing constructions. Moreover, in Section~\ref{generalization} we investigate some properties of $(w,d)$-covering of graphs and determine a relationship between this parameter and the biclique covering number of bipartite graphs.
\section{Cover Free-Family}{} \label{bc and bp section}
Determining the exact value of the parameter $N((r;w; d); t)$ and the biclique covering number of the Kneser graphs, even for special $r$, $w$, $d$, and $t$, is an interesting
and challenging problem. This problem has been studied in the literature; see~\cite{haji2,Hajiabolhassan20123626,kim2,kim,li}.
In this section, we determine the exact value of $N((r;w; d); t)$ for some special values of $r$, $w$, $d$, and $t$. Also, 
we determine the exact value of the biclique covering number of some Kneser graphs. To do this, we present a preliminary lemma as follows.
\begin{lem} \label{big bc} Let $k$ and $t$ be positive integers, where $2\leq k\leq t$. Then
$$B(KG(2t,k))= B(I_{2t}(k,k)) ={t\choose k}^{2},$$
where $B(G) $ is the maximum number of edges among the bicliques of $G$.
\end{lem}
\begin{proof}{One can check that the maximum number of edges among the bicliques of each of these graphs is 
$$ \mathop{\max}_{k \leq i \leq t} [{  i \choose  k } { 2t-i \choose k }]. $$
Let $ f(i)={i \choose k} {2t-i \choose k}$, where $k \leq i\leq t$. Since $\frac{f(i-1)}{f(i)} =\frac{(i-k)(2t-i+1)}{i(2t-i-k+1)} $, it is easy to check that  $f$ is an increasing function. Hence,
$$ B(KG(2t,k)) = B(I_{2t}(k,k)) = {  t \choose k }^{2},$$
as desired.}
\end{proof}
\begin{thm}\label{bi intersect} If $ 1 \leq k \leq t$ and $ d={ 2t-2k \choose t-k}$, then
 $$N((k,k;d),2t)=bc_{d}(I_{2t}(k,k))=bp_{d}(I_{2t}(k,k))={ 2t \choose t}.$$
\end{thm}
\begin{proof}{Define $t'$ to be $ { 2t \choose t } $. Now, we show that $ I_{2t}(k,k) $ can be covered by $ t' $ bicliques such that every edge of $ I_{2t}(k,k) $ is covered by exactly $d$ bicliques. Denote the vertex set of $ I_{2t}(k,k) $ by bipartition $ (X,Y ) $ 
in which $X$ and $Y$ are the collections of all $t$-subsets of the set $[2t]$.
Suppose $ A_{j} $ is a $ t $-subset of $[2t] $ and $ A_{j}^{c} $ is the complement of the $ A_{j} $ in  $ [2t] $. 
Denote the number of  these pairs by $ t' $. Now, for every $ 1 \leq j \leq  t'  $, construct the biclique $ G_{j} $ with vertex set 
$ (X_{j} , Y_{j}) $, where $X_{j}$ is all $k$-subsets of $A_{j}$ and $Y_{j}$ is all $k$-subsets of $A_{j}^{c}$.
Let $UV$ be an arbitrary edge of $ I_{2t}(k,k)$.
In view of the definition of  $G_j$, $UV$ is covered by every  $G_{j}$ with vertex set $(X_{j} , Y_{j})$, where $U$ is a vertex of $X_{j}$  and $V$ is a vertex  of $Y_{j}$ or vice versa. Thus every edge of $ I_{2t}(k,k)$ is covered by at least $ d $ bicliques. One can see that 
$$\sum_{j=1}^{t'}|E(G_{j})| = {2t \choose t}{ t \choose k }^2 \quad \quad \& \quad \quad \vert E( I_{2t}(k,k)) \vert = {2t \choose k}{ 2t-k \choose k}. $$
Now, it is simple to check that
$$\sum_{j=1}^{t'}|E(G_{j})| = d \vert E( I_{2t}(k,k)) \vert.$$
Thus every edge of $I_{2t}(k,k)$ is covered by exactly $ d $ bicliques. Note that we have actually proved that 
\begin{equation}\label{eq-bi intersect1}
bp_{d}(I_{2t}(k,k)) \leq t'.
\end{equation}
Conversely, one can see that 
$$bp_{d}(I_{2t}(k,k)) \geq bc_{d}(I_{2t}(k,k)) \geq \frac{d|E(I_{2t}(k,k))|}{B(I_{2t}(k,k))}.$$
Also, by Lemma ~\ref{big bc}, we have
$$\frac{d|E(I_{2t}(k,k))|}{B(I_{2t}(k,k))}=\frac{{ 2t-2k\choose t-k }{ 2t \choose k }{2t-k\choose k }}{{ t\choose k}^2}={ 2t\choose t }= t'. $$
Hence, 
\begin{equation}\label{eq-bi intersect2}
bp_{d}(I_{2t}(k,k)) \geq bc_{d}(I_{2t}(k,k)) \geq t'.
\end{equation}
From ~(\ref{eq-bi intersect1}) and ~(\ref{eq-bi intersect2}) we conclude
$$bp_{d}(I_{2t}(k,k)) = bc_{d}(I_{2t}(k,k)) = t',$$
which completes the proof. 
}
\end{proof}
In view of the proof of the previous theorem, by a slight modification, one can obtain the following result.
\begin{thm}\label{bc of Kneser}
If $1\leq k\leq t$ and $ d={2t-2k\choose t-k }$, then
 $$bc_d(KG(2t,k))=bp_d(KG(2t,k))=\frac{{2t\choose t}}{2}.$$
\end{thm}
There are many applications for cover-free families. Consequently, the efficient construction of cover-free families is an interesting problem for researchers. Li  et al.~\cite{li} proved the following theorem.
\begin{alphthm}\label{li}{\rm \cite{li}}
If there exists a $(2,1;d)-CFF(n,t)$, then there exists a $(2,1;d)-CFF(n+(s+2)(d+1),2t)$, where $s=N((1,1),t)$.
\end{alphthm}
In the next theorem, we give a construction which  improves Theorem~\ref{li}.
\begin{thm}\label{new}If there exists a $(2,1;d)-CFF(n_1,t)$ and a $(1,1;d)-CFF(n_2,t)$, 
then there exists a $(2,1;d)-CFF(n_1+n_2+2,2t)$.
\end{thm}
\begin{proof} { Let $\mathcal{P} =\{G_1,\ldots, G_{n_1}\}$ and $\mathcal{A}=\{K_1, \dots, K_{n_2}\}$ be a $d$-biclique cover of $I_t(1,2)$ and $I_t(1,1)$, respectively. Also, assume that $G_i$ and $K_i$ have $(X_i, Y_i)$ and $(W_i, Z_i)$ as vertex sets, respectively.  For $i=1, \ldots, n_1$,  let $E_i$ be the union of all $2$-subsets of $Y_i$.  Suppose that 
$ \{1, 2, \ldots, t, 1', 2', \ldots, t'\}$ is the ground set of the vertex set of the graph $I_{2t}(1,2)$.  In the sequel, for any $A\subseteq [t]$, $A'$ stands for the set $\{ i'  |  i \in A\}$. Let $A_i$ and $C_i$ be  the set of all  $1$-subsets of the set $X_i\cup X'_i$ and $Z_i \cup W'_i$, respectively. Also, let $B_i$ and $D_i$ be the set of all $2$-subsets 
 of the set $E_i \cup E'_i$ and $W_i \cup Z'_i$, respectively.  Let $H_i$ and $L_i$ be two bicliques in which $(A_i, B_i)$ and $(C_i, D_i)$ are the bipartition of their vertex sets,  respectively. Also, let $H$ be a biclique that has $(A, B)$ as the bipartition of its vertex set, where $A$ contains all $1$-subsets of the set $[t]$ and $B$ contains all $2$-subsets of the set $\{1', \ldots, t'\}$. Similarly, let $K$ be a biclique with $(Z, W)$ as the bipartition of its vertex set, where $Z$ contains all $1$-subsets of the set $\{1', \ldots, t'\}$ and $W$ contains all $2$-subsets of the set $[t]$. Easily one can check  that the set $\{H_1, \ldots, H_{n_1}, L_1, \dots, L_{n_2}, H, K\}$ form a $d$-biclique cover for the graph $I_{2t}(1,2)$.}
\end{proof}
The proof of the next theorem is analogous to that of Theorem~\ref{new}.
\begin{thm}If there exists a $(2,2;d)-CFF(n_1,t)$ and a $(2,1;d)-CFF(n_2,t)$, then there exists a $(2,2;d)-CFF(n_1+2n_2+2,2t)$.
\end{thm}
\section{Generalization of Cover-Free Families}\label{generalization}
Assume that $G$ is a graph with $t$ vertices and $\delta(G)> w$.
Let ${\cal A}= \{A_1,A_2,\ldots, A_n\}$ be an optimal covering
of $G$. Assume that $v$ is an arbitrary vertex of the graph $G$ and $W$ is a $w$-subset of vertices such that $v \not\in W$. 
Since $\delta(G)> w$, 
there exists a vertex $u$ adjacent to $v$ such that $u \not\in W$. In view of the definition of a covering, there exists $\{i_1, \ldots, i_d \}$ 
such that, for each $k = 1,2,\ldots, d$, $\{u,v\} \subseteq A_{i_k}$ and  $W\cap A_{i_k} = \varnothing$. So we have
 $N((1,w;d),t) \leq N(G,w;d) $.
\begin{lem} \label{low-bound} If $G$ is a graph with $m$ edges, then $$ m \leq T((1,w;d),N(G,w;d)).$$
\end{lem}
\begin{proof} { Let ${\cal A} = \{A_1,A_2, \ldots, A_n \} $ be an optimal  $(w,d)$-covering of $G$, i.e., $n=N(G,w;d)$. Set
$${\cal B} =\{ K_i \cap K_j \ | \  ij \in E(G) \}, $$
where $K_i$ is the set of keys of the $i^{\rm th}$ user. It is easy to see that $ {\cal B} $ is a $(1,w;d)-CFF$ with $m$ blocks and $n$ elements . Hence, $m \leq T((1,w;d),n)$.}
\end{proof}
An antichain $ \{A_1,A_2,...,A_t\} $ on a set $A$ is a family of nonempty subsets of $A$ such that $A_i \subseteq A_j$ implies that $i=j$. In fact, an antichain with $t$ block on the set $[n]$ is a $(1,1)-CFF(n,t)$. By Sperner's lemma,  if ${\cal A} = \{ A_1,A_2,\ldots,A_t \} $ is an antichain on the set $[n]$, then
$t \leq { n \choose \lfloor \frac{n}{2} \rfloor}$.  Hence, $T((1,1); n) \leq { n \choose \lfloor \frac{n}{2} \rfloor}$ and if ${\cal R}(t)=\min\{ c \,\ | \,\ {c
\choose \lfloor \frac{c}{2}\rfloor}\geq t \}$, then ${\cal R}(t)=N((1,1),t)$. Note that, by Sterling's formula, ${\cal R}(t)= \log_2 t +\frac{1}{2}\log_2\log_2 t + O(1)$. Erd\"{o}s et al.~\cite{erdos1} discussed $(2,1)$-CFFs in detail, and showed that
$$1.134^n \leq T((2,1),n) \leq 1.25^n.$$
The upper bound is asymptotic and for sufficiently large $n$ is useful.
Here is the best known lower bound for $N((1,w),t)$.
\begin{alphthm}\label{rcff}{\rm\cite{rcff1, rcff2, rcff3}} Let $w \geq 2$ and $t \geq w+1$ be  positive
integers. Then
$$N((1,w),t)\geq C_{w,t}\frac{w^2}{\log w}\log t,$$
where $\displaystyle \lim_{w+t\rightarrow \infty}C_{_{w,t}}=c$ for some constant $c$.
\end{alphthm}
In \cite{rcff1, rcff2, rcff3}, it was shown that $c$ is approximately $\frac{1}{2}$,
$\frac{1}{4}$, and  $\frac{1}{8}$, respectively. As a result of this lower bound, we have
$$T((1,w),n) \leq w^{\frac{n}{c w^2}}.$$
So the following corollary is concluded.
\begin{cor}\label{lower} If $G$ is a graph with $m$ edges, then
\begin{enumerate}
\item $\frac{2}{1+ \log_2e}\log_2 m \leq N(G)$,
\item $\frac{1}{\log 1.25}\log m \leq N(G,2;1)$,
\item $ c\frac{w^2}{\log w}\log m \leq N(G,w;1)$, for every $w \geq 2$.
\end{enumerate}
\end{cor}
Let $K_{t_1,t_2}$ be the complete bipartite graph. In the next proposition, we give an upper bound for the $(1,d)$-covering of these graphs.
\begin{pro}\label{low-up bound2} If $t_1$, $t_2$, and $d$ are positive integers, then
$$ N(K_{t_1,t_2},1;d)\leq N((1,1;d),t_1)+N((1,1;d),t_2).$$
\end{pro}
\begin{proof}{Label the vertices of the first and second part of $K_{t_1,t_2}$ with
$v_1,v_2, \ldots, v_{t_1}$ and $u_1,u_2, \ldots, u_{t_2}$, respectively.
Assume that $\{A_1,A_2,\ldots, A_{n_1}\}$ and $\{S_1,S_2,\ldots,
S_{n_2}\}$ are optimal $(1,d)$-covering of the complete graph with $t_1$ vertices and $t_2$ vertices respectively, i.e., 
$n_1=N((1,1;d),t_1)$ and $n_2= N((1,1;d),t_2)$. Define $A'_i=A_i\cup \{u_1,u_2,
\ldots, u_{t_1}\}$ for $i=1,2, \ldots, n_1$ and $S'_i=S_i\cup
\{v_1,v_2, \ldots, v_{t_2}\}$ for $i=1,2, \ldots, n_2$. One can 
check that the collection $\{A'_1, \ldots, A'_{n_1},S'_1, \ldots,
S'_{n_2}\}$ is a covering of the graph $K_{t_1,t_2}$ and so
$$N(K_{t_1,t_2},1;d)\leq N((1,1;d),t_1)+N((1,1;d),t_2),$$
as desired. }
\end{proof}
By the lower bound of Corollary~\ref{lower} and the upper bound of Proposition~\ref{low-up bound2}, for every positive integer $t$,  we have
$$(1.637) \log_2 t\leq N(K_{t,t})\leq 2N((1,1),t)=2{\cal R}(t)=2 \log_2 t +\log_2\log_2 t + O(1).$$
If for a graph $G$, there exists a covering of the edges of the complete graph $K_t$ with $l$ copies of $G$,
then one can see that there exists a $(2,w;d)-CFF(lN(G,w;d),t)$. In \cite{katona}, Katona and Szemer\'{e}di showed that the edges of the complete graph $K_t$ can be covered by the complete bipartite graph $K_{\frac{t}{2}, \frac{t}{2}}$ with a collection of size $\log_2 t$. So we have the following corollary.
\begin{cor} There exists  a $(2,1)-CFF(2{\cal R}(\frac{t}{2})\log_2t,t)$.
\end{cor}
 In the next proposition, we show the relationship between  $N(G)$ and biclique covering number of a special families of bipartite graphs.
\begin{pro} \label{key pool number} Let $G$ be a graph. There exists a bipartite graph $H$ such that 
$$N(G,w;d)=bc_d(H).$$
\end{pro}
\begin{proof}{ Let $H$ be a bipartite graph whose vertices are all edges and all $w$-subsets of vertices of the graph $G$. We say an
edge, $e=\{u,v \}$, is incident to a $w$-subset $W$ if and only if $W \cap \{u, v \}=\varnothing$. We claim that $N(G,w;d)=bc_d(H).$ To see this, first assume that the collection $\{ G_1, \ldots, G_l \}$ is a $d$-biclique cover of the graph $H$, where $l=bc_d(H)$ and $G_i$ has $(X_i, Y_i)$ as its vertex set.  
 Let $A_i$ be the union of the vertices of  edges lying in $X_i$ and $B_i$ be the union of the vertices of  $w$-subsets lying in $X_i$. Set ${\cal A}= \{ A_1, A_2, \ldots, A_l \}$. It is easy to check that ${\cal A}$ is a $(w,d)$-covering of the graph $G$. So $N(G,w;d) \leq bc_d(H)$. Conversely, assume that  ${\cal A}= \{ A_1, A_2, \ldots, A_l \}$ is an optimal $(w,d)$-covering of the graph $G$, i.e., $l=N(G,w;d)$.  
Now, for any $1\leq j \leq l$, construct a bipartite graph $G_j$ with the vertex set $(X_j,Y_j)$,  where the
vertices of $X_j$ are all edges of $G$ that its end points are elements of $A_j$ and the vertices
of $Y_j$ are all $w$-subsets of the set $A_j^c$. Also, an edge is incident to a $w$-subset  if their
intersection is empty. So $G_j$ is a complete bipartite subgraph of $H$.  Let $UV$ be an arbitrary edge of the graph $H$. Hence, there exists an edge  $e=\{ u, v\}$ of the graph $G$ and a $w$-subset $W$ of vertices  disjoint from $\{ u, v\}$ such that $U=\{ u, v\}$ and $V=W$. Since ${\cal A}$ is a $(w,d)$-covering of the graph $G$, there exist $d$ indices $i_1,\ldots,i_d$ such that  $e=\{u, v \} \subseteq A_{i_j}$, and $W \cap  A_{i_j}= \varnothing$,  for $j=1, \dots, d$. Therefore, $G_{i_1}, \ldots G_{i_d}$ covers $UV$ and  $\{G_1,G_2, \ldots, G_l\}$ is a
$d$-biclique cover of $H$. So $bc_d(H) \leq N(G,w;d)$.
}
\end{proof}
Let  $K_{t,t}^-$ be the graph $K_{t,t}$ with a perfect matching removed.  Determining the biclique covering number of  the $K_{t,t}^-$ was  discussed by Bezrukov et al.~\cite{Bezrukov2008319}.  Let $G$ be a graph and $S$ be a subset of the edges of $G$. The graph $G \setminus S$ is obtained from $G$ by removing $S$.
By the proof of Theorem~\ref{key pool number}, for every graph $G$ with $t$ vertices and $m$ edges, we have 
$$N(G)= bc(K_{m,t}\setminus K),$$ 
where $K$ is a bipartite graph in which every vertex in the first part has degree $2$ and
every vertex in the second part has the same degree as the graph $G$. So we have the following corollary.
\begin{cor}\label{cycle}For every integer $n$, $N(C_n)= bc(K_{n,n}\setminus C_{2n})$.
\end{cor}
\begin{lem}\label{hom}  Assume that $G_1, G_2, \ldots, G_k$ are graphs such that for any $i=2, \ldots, k$, the graph $G_1$ contains a subgraph isomorphic to $G_i$. Then
$$N(\cup_{i=1}^k G_i) \leq N(G_1)+k.$$
\end{lem}
\begin{proof}{ Let ${\cal A}=\{A_1, \ldots, A_n\}$ be an optimal covering of the graph $G_1$. For any $j\in \{2,\ldots, k\}$, assume that $f_j:V(G_j) \rightarrow V(G_1)$ is an injective homomorphism, i.e., a one-to-one map which preserves the adjacency. For every $1 \leq i \leq n$, set  $B_i =\{v\in V(G_j) \ | \  j=2,\ldots k,\ f_j(v) \in A_i\}$ and $C_i= A_i \cup B_i$.
One can check that the collection  ${\cal C}=\{ C_1, \ldots, C_n, V(G_1), \ldots, V(G_k) \} $ is a covering for the graph $\cup_{i=1}^k G_i$.
}
\end{proof}
\begin{lem}\label{star covering} 
If $G$ is a star graph with $t+1$ vertices, then $$N(G,w;d)=N((1,w;d),t).$$
\end{lem}
\begin{proof}{Let $G$ be a star which $v$ is the interval vertex  and
$v_1,v_2, \ldots, v_t$ are it's leaves. Let ${\cal A}= 
\{A_1,A_2,\ldots, A_n\}$ be a minimum $(w,d)$-covering of $G$. Since
${\cal A}$ is a minimum covering, every $A_i$ contains the vertex
$v$. Consider the collection ${\cal S}=\{ S_i\,\ | \,\ S_i=A_i\setminus \{ v \}, \,\,\,\
i=1,\ldots, n \}$. One can see that  the collection ${\cal S}$ is a $(w,d)$-covering of a complete graph with $t$ vertices. So, $N((1,w;d),t) \leq N(G,w;d)$. 
Conversely, consider an optimal $(1,w;d)-CFF(n,t)$ i.e, $n=N((1,w;d),t)$. Set ${\cal S}= \{S_1,S_2,\ldots, S_n\}$, where $S_i$ is the set of users that have $i^{\rm th}$ key. One can check that 
the collection $\{A_i \, | \, A_i=S_i \cup\{v\}, \ i=1, \ldots,n \}$ is a
$(w,d)$-covering of $G$. Therefore, $N(G,w;d)\leq N((1,w;d),t)$.}
\end{proof}
\begin{thm}\label{CN_tree1}
Let $T$ be a tree with $ m $ edges and maximum degree $ \Delta $.
 Also, assume that $ t $ is the number of vertices of $T$ whose degrees are at least $ 3 $. Then 
$$ N(T)\leq 2\lceil \log_2m\rceil+{\cal R}(\Delta) + t.$$
\end{thm}
\begin{proof}{ We prove the theorem in two steps. First suppose that $T$ is a path of length $n$. Assume that $k$ is the smallest positive integer such that $2^{k} \geq n$. It is sufficient to prove the proposition for the path $P=P_{2^{k}+1}$. We decompose $P$ into two edge-disjoint paths $P_{1,1} , P_{1,2}$ with the same length. Inductively, for each $ i $ such that $1 \leq i \leq k-1 , 1 \leq j \leq 2^{i}$,
we decompose the path $P_{i,j}$ into two edge-disjoint paths $P_{i+1,2j-1} , P_{i+1,2j}$  with the same length. Set $A_{1}=V(P_{1,1})$ and $B_{1}=V(P_{1,2})$. Also, 
for each $ i $ such that  $2 \leq i \leq k-1 $, define  
$$A_{i}= \bigcup_{t=1}^{2^{i-2}} V(P_{i,4t}) \bigcup_{t=1}^{2^{i-2}} V(P_{i,4t-3})\quad \& \quad B_{i}= \bigcup_{t=1}^{2^{i-2}} V(P_{i,4t-2}) \bigcup_{t=1}^{2^{i-2}} V(P_{i,4t-1}),$$
where $ V(P_{i,j}) $ is the vertex set of $ P_{i,j} $. Now, we show that 
$$ {\cal A}=\lbrace A_{1}, ..., A_{k}, B_{1}, ..., B_{k} \rbrace$$
is a covering of the path $P$. Assume that $uv$ is an arbitrary edge of $P$ and $ w $ is a  vertex of $P$ other than $ u $ and $v$. In view of the definition of $P_{i,j}$'s, there exists a positive integer 
$t$ such that $uv\in E(P_{t,j})$ and $w\not \in V(P_{t,j})$. One can check that either $uv\in A_t$ and $w\not \in A_t$ or $uv\in B_t$ and $w\not \in B_t$, as desired.
Let us now assume that $ T $ is a union of $l$ paths where $l\geq 1 $ and $ {\cal V' } = \{ v_1, v_2, ... ,v_t \} $ is the collection of vertices of $T$ whose degrees are at least $ 3 $. Let  $ T'= T \setminus V' $. One can check that $T'$ is a union of vertex-disjoint paths with at most $ m - 3t $ vertices. According to the first step of the proof, $T'$  has a covering of size $ 2\lceil log_2(m-3t) \rceil $.  Let $ S_i $ be the subgraph induced by the edges of $T$ incident to $v_i$,  for every $ i $ where $ 1\leq i \leq t$. Since $ S_i $, for every  $ 1\leq i \leq t$, is an star, by  Lemmas~\ref{hom} and~\ref{star covering}, the subgraph  induced by $ \bigcup_{1 \leq i \leq t}^{} S_{i} $ can be covered by a collection of size at most  ${\cal R}(\Delta) +t$. This completes the proof of theorem.
}
\end{proof}
By writing out a proof similar to that of Theorem~\ref{CN_tree1} and by Corollary~\ref{cycle}, we obtain  the following result.
\begin{cor} \label{CN of cycle}
If $ C_n $ is  a cycle of lenght $ n $, then 
$$ {\cal R}(n) \leq bc(K_{n,n}\setminus C_{2n})=N(C_n )\leq 2\lceil \log_2n \rceil + 1.$$
\end{cor}
The {\it Cartesian product} of two graphs $G$ and $H$ denoted by $G\square H$ is a graph such that the vertices of $G\square H$ is the set $V(G)\times V(H)$ and two vertices $(u,u')$ and $(v,v')$ are adjacent if and only if either $u=v$ and $u'$ is adjacent to $v'$ in the graph $H$ or $u'=v'$ and $u$ is adjacent with $v$ in the graph $G$. 
\begin{cor} \label{CN of lattice}  If $t_1$ and $t_2$ are positive integers, then
$$ N(P_{t_1} \square P_{t_2}) \leq  2\lceil \log_2t_1t_2 \rceil + 2.$$
\end{cor}
\begin{proof} { First we prove that if $G$ and $H$ are two graphs such that $\delta(G)\geq 2$ and  $\delta(H)\geq 2$, then $ N(G\square H)\leq N(G) + N(H).$ To see this, One can check that if $\{ A_1, A_2, \ldots, A_{n_1} \}$ and $\{ B_1, B_2, \ldots, B_{n_2} \}$  are coverings of $G$ and $H$ respectively, then $\{ A_1 \times V(H), A_2 \times V(H), \ldots, A_{n_1} \times V(H) \} \cup \{ V(G) \times B_1, V(G)\times B_2, \ldots, V(G) \times B_{n_2} \}$ is a covering for the graph $G\square H$. So $N(C_{t_1} \square C_{t_2}) \leq  2\lceil \log_2t_1t_2 \rceil + 2 $. Since $P_{t_1} \square P_{t_2}$ is a subgraph of $C_{t_1} \square C_{t_2}$, the result follows. }
\end{proof}
\begin{alphthm}\label{Lovasz LL}{\rm \bf (Lovasz Local Lemma)} Suppose that $ A_1, ... ,A_l $ are events in a probability space with $Pr[A_i] \leq p  $ for all $ i $. If each event is mutually independent of all the other events except for at most $ d $ of them, and if $ ep(d+1)\leq 1 $, then $ Pr[\bigcap_{i=1}^{l}  \bar A_l]>0 $.
\end{alphthm}
In view of the Lovasz Local Lemma, we give an  upper bound for $ N(G,w;1) $. 
\begin{thm}\label{upper bound1}
If $G$ is a graph with $t$ vertices, and $m$ edges,  then
$$ N(G,w;1)\leq \lceil\frac{\log_2{(e[D+1]) }}{-\log_2 q}\rceil\ ,$$
 where
$$q=1- p^{2}(1-p)^{w}\quad , \quad 0<p<1$$
$$D=m {t-2 \choose  w } - (m -(w+2)\Delta){ t-w-4 \choose w }$$
\end{thm}
\begin{proof}{ Let $A =[a_{ij}]$ be an  $N \times t$ random matrix whose entries are mutually independent chosen from $ \{ 0,1 \} $  such that $Pr(a_{ij} =1)=p$, $N$ to be determined. The columns of $A$ are labelled  by the vertices of $G$ and assume that $V(G)=\{ u_1, \cdots, u_t\}$. Assign to the  $i^{\rm th}$ row of $A$, the subset $A_i$ of the vertices of $G$ as follows
$$ A_i=\{u_ j | 1 \leqslant j \leqslant t, a_{ij} = 1 \}.$$
For every edge $U= \lbrace u_{1}, u_{2}\rbrace$ and a subset $ W $ of vertices, where  $W \subseteq V(G)\setminus U $ and $ \vert W \vert = w $, let $A_{(U,W)}$ be the event that there does~not exist a row of $ A $ such that all entries in the columns in $ U $  are $1$  and all entries in the columns in $ W $ are $0$. Note that ${\cal A}=\lbrace A_{1},A_{2},...,A_{N} \rbrace$ is a $(w;1)$-covering of $G$ if and only if none of these events occur, that is, if $Pr(\cap \bar A_{(U,W)})>0$.
One can check that
$$Pr(A_{(U,W)})=q^{N},$$
  where
  $$q=1- p^{2}(1-p)^{w} \quad \& \quad 0<p<1 .$$
  Note that the event $A_{(U,W)}$  is mutually independent of all the other events  $A_{(U',W')}$  except those with
$$( U  \cup  W ) \cap ( U'  \cup  W' ) \neq \phi.$$
 There are at most
$$ D=m { t-2 \choose w } -  (m -(w+2)\Delta){ t-w-4\choose w }$$
such events.
According to the Lovasz Local Lemma, a $(w;1)$-covering of G exists whenever
$$ e(D+1)q^{N} \leq 1 .$$
 Taking logarithms of both sides of this inequality and rearranging, we get the desired
result.}
\end{proof}
{\bf Acknowledgments:} This paper is a part of Mehdi Azadi Motlagh's Ph.D. Thesis. 
The authors would like to express their deepest gratitude to Professor Hossein Hajiabolhassan to introduce 
a generalization of cover-free families and also for his invaluable comments and discussion. 

\begin{thebibliography}{10}

\bibitem{Bezrukov2008319}
Sergei Bezrukov, Dalibor Fronček, Steven~J. Rosenberg, and Petr Kovář.
\newblock On biclique coverings.
\newblock {\em Discrete Mathematics}, 308(2–3):319 -- 323, 2008.

\bibitem{bllobas1}
B{\'e}la Bollob{\'a}s and Alex Scott.
\newblock Separating systems and oriented graphs of diameter two.
\newblock {\em J. Combin. Theory Ser. B}, 97(2):193--203, 2007.

\bibitem{rcff1}
A.~G. Dyachkov and V.~V. Rykov.
\newblock Bounds on the length of disjunctive codes.
\newblock {\em Problemy Peredachi Informatsii}, 18(3):7--13, 1982.

\bibitem{erdos1}
P.~Erd{\H{o}}s, P.~Frankl, and Z.~F{\"u}redi.
\newblock Families of finite sets in which no set is covered by the union of
  two others.
\newblock {\em J. Combin. Theory Ser. A}, 33(2):158--166, 1982.

\bibitem{rcff2}
Zolt{\'a}n F{\"u}redi.
\newblock On {$r$}-cover-free families.
\newblock {\em J. Combin. Theory Ser. A}, 73(1):172--173, 1996.

\bibitem{haji2}
Hossein Hajiabolhassan and Farokhlagha Moazami.
\newblock Secure frameproof codes through biclique covers.
\newblock {\em Discrete Math. Theor. Comput. Sci.}, 14(2):261--270, 2012.

\bibitem{Hajiabolhassan20123626}
Hossein Hajiabolhassan and Farokhlagha Moazami.
\newblock Some new bounds for cover-free families through biclique covers.
\newblock {\em Discrete Mathematics}, 312(24):3626 -- 3635, 2012.

\bibitem{katona}
G.~Katona and E.~Szemer\'{e}di.
\newblock On a problem of graph theory.
\newblock {\em Studia Sci.Math. Hungar.}, pages 23--28, 1967.

\bibitem{1053689}
W.~Kautz and R.~Singleton.
\newblock Nonrandom binary superimposed codes.
\newblock {\em Information Theory, IEEE Transactions on}, 10(4):363--377, Oct
  1964.

\bibitem{JCD:JCD10056}
Hyun~Kwang Kim and Vladimir Lebedev.
\newblock On optimal superimposed codes.
\newblock {\em Journal of Combinatorial Designs}, 12(2):79--91, 2004.

\bibitem{kim2}
Hyun~Kwang Kim, Vladimir Lebedev, and Dong~Yeol Oh.
\newblock Some new results on superimposed codes.
\newblock {\em J. Combin. Des.}, 13(4):276--285, 2005.

\bibitem{kim}
Sh.~Kh. Kim and V.~S. Lebedev.
\newblock On the optimality of trivial {$(w,r)$}-cover-free codes.
\newblock {\em Problemy Peredachi Informatsii}, 40(3):13--20, 2004.

\bibitem{li}
P.~C. Li, G.~H.~J. van Rees, and R.~Wei.
\newblock Constructions of 2-cover-free families and related separating hash
  families.
\newblock {\em J. Combin. Des.}, 14(6):423--440, 2006.

\bibitem{rcff3}
Mikl{\'o}s Ruszink{\'o}.
\newblock On the upper bound of the size of the {$r$}-cover-free families.
\newblock {\em J. Combin. Theory Ser. A}, 66(2):302--310, 1994.

\bibitem{frame}
D.~R. Stinson, Tran van Trung, and R.~Wei.
\newblock Secure frameproof codes, key distribution patterns, group testing
  algorithms and related structures.
\newblock {\em J. Statist. Plann. Inference}, 86(2):595--617, 2000.
\newblock Special issue in honor of Professor Ralph Stanton.

\bibitem{stinson1998}
Douglas~R Stinson and Ruizhong Wei.
\newblock Combinatorial properties and constructions of traceability schemes
  and frameproof codes.
\newblock {\em SIAM Journal on Discrete Mathematics}, 11(1):41--53, 1998.

\bibitem{Stinson2004463}
D.R. Stinson and R.~Wei.
\newblock Generalized cover-free families.
\newblock {\em Discrete Mathematics}, 279(1–3):463 -- 477, 2004.

\bibitem{wei}
R.~Wei.
\newblock On cover-free families.
\newblock {\em manuscript.}, 2006.

\end{thebibliography}
\def\cprime{$'$}

\end{document}